\documentclass[12pt]{amsart}
\usepackage[a4paper]{geometry}
\usepackage{tabularx}

\usepackage[all]{xy}
\usepackage{comment}
\usepackage{graphicx} 
\usepackage[T1]{fontenc}
\usepackage{hyperref}
\usepackage{array}   
    \newcolumntype{L}{>{$}l<{$}} 
\usepackage{amssymb}
\usepackage{amsthm}
 \usepackage{amsmath}
\usepackage{bbm}
\usepackage{mathrsfs}
\usepackage{float}
\usepackage{enumitem}
\usepackage{comment}

\newtheorem{theorem}{Theorem}[section]
\newtheorem{proposition}[theorem]{Proposition}
\usepackage[dvipsnames]{xcolor}

\newtheorem{lemma}[theorem]{Lemma}
\theoremstyle{definition}

\theoremstyle{remark}

\usepackage{MnSymbol}

\usepackage{dsfont}

\newcommand{\ind}{\mathbbm{1}}

\newcommand{\bxi}{\boldsymbol{\xi}}
\newcommand{\calP}{\mathcal{P}}
\newcommand{\calM}{\mathcal{M}}
\newcommand{\calQ}{\mathcal{Q}}

\newcommand{\FF}{\mathbb{F}}
\newcommand{\EE}{\mathbb{E}}
\newcommand{\ZZ}{\mathbb{Z}}
\newcommand{\CC}{\mathbb{C}}

\newcommand{\PP}{\mathbb{P}}

\newcommand{\MM}{\mathcal{M}}
\newcommand{\bF}{\boldsymbol{F}}
\newcommand{\bG}{\boldsymbol{G}}
\newcommand{\bH}{\boldsymbol{H}}
\newcommand{\be}{e_{\calP}}

\newcommand{\sumP}[1]{\sum_{\boldsymbol{#1} \in \mathcal{M}_{\calP,n}}}

\newcommand{\ps}[1]{\left( #1 \right)}

\DeclareMathOperator{\disc}{\mathrm{disc}}

\DeclareMathOperator{\res}{\mathrm{res}}

\title{Full Galois groups of polynomials with slowly growing coefficients}
\author{Lior Bary-Soroker}
\address{School of Mathematical Sciences, Tel Aviv University, Tel Aviv 69978, Israel}
\email{barylior@tauex.tau.ac.il}
\author{Noam Goldgraber}
\address{School of Mathematical Sciences, Tel Aviv University, Tel Aviv 69978, Israel}
\email{noam3goldgraber@gmail.com}

\date{\today}
\begin{document}

\begin{abstract}
    Choose a polynomial $f$ uniformly at random from the set of all monic polynomials of degree $n$ with integer coefficients in the box $[-L,L]^n$. The main result of the paper asserts that if $L=L(n)$ grows to infinity, then the Galois group of $f$ is the full symmetric group, asymptotically almost surely, as $n\to \infty$. 
    
    When $L$ grows rapidly to infinity, say $L>n^7$, this theorem follows from a result of Gallagher. When $L$ is bounded, the analog of the theorem is open, while  the state-of-the-art is that the Galois group is large in the sense that it contains the alternating group (if $L< 17$, it is conditional on the Extended Riemann Hypothesis). Hence the most interesting case of the theorem is when $L$ grows slowly to infinity.

    Our method works for more general independent coefficients.
\end{abstract}

\maketitle

\section{Introduction}
The study of the Galois group of a random polynomial is going back to the foundational works of Hilbert \cite{Hilbert} and van der Waerden \cite{VanDerWaerden}. Expressed in terms of probability theory, they proved that if we uniformly choose at random a monic polynomial $f$ of degree $n$ whose coefficients are within the box $([-L,L]\cap \mathbb{Z})^n$, then its Galois group $G_f$ is the full symmetric group, almost surely as $L\to \infty$ and $n$ is fixed; i.e.\ $\displaystyle\lim_{L\to \infty} \PP(G_f=S_n)=1$. We call this model the \emph{large box model}.

Van der Waerden conjectured that the second most probable group in the large box model is $S_{n-1}$ coming from polynomials having a rational root. Chela \cite[Theorem~1]{Chela} computed that $\PP(G_f=S_{n-1})\sim \frac{c_n}{L}$, as $L\to \infty$ and $n>2$, where  $c_n=\Theta(2^n)$ is an explicit constant.  Thus, the van der Waerden conjecture may be stated as $\PP(G_f\neq S_{n})\sim \frac{c_n}{L}$, $L\to \infty$.   A slightly weaker version of the conjecture is that $\PP(G_f
\neq S_n) = O(L^{-1})$, as $n$ is fixed and $L\to\infty$. The weaker version is occasionally also referred to as the van der Waerden conjecture. 

After almost of a century of progress \cite{Andersonetal,ChowDietmann,ChowDietmann2,Dietmann_2011,Gallagher_1976,Knobloch,VanDerWaerden}, the latter bound was established by Bhargava \cite{Bhargava}.
The main challenge in the above results is to bound $\PP(G_f=A_n)$ from above, in particular, a key result in \cite{Bhargava} is that 
\[
    \mathbb{P}(G_f= A_n)=O(L^{-1}), \qquad L\to\infty.
\]
This is the state-of-the-art, even though it is believed to be not sharp, see \cite[Conjecture 1.1]{bary2022probabilistic}.

In the large box model, the degree $n=\deg f$ is fixed. In particular, the rate of convergence and the implied constants depend on $n$. Indeed, in the majority of the results mentioned above, the implied constant is at least exponential, sometimes super-exponential in $n$, or not given explicitly. A notable exception is the method of  Gallagher \cite{Gallagher_1976} that is based on the large sieve and gives a polynomial dependence on $n$: 
\begin{equation}\label{GallagherThm}
    \mathbb{P}(G_f \neq S_n)\ll \frac{n^3 \log L}{\sqrt{L}},
\end{equation}
where the implied constant is absolute, see also \cite[Theorem 4.2]{Kowalski_2008}.

Random polynomials are central in probability theory. One of the most  natural and well-studied models are when the coefficients are sampled independently and the degree goes to infinity. This model  goes back to the seminal works of Bloch-P\'olya \cite{BlochPolya}, Littlewood-Offord \cite{LittlewoodOfford}, Kac \cite{Kac}, and Erd\H{o}s-T\'uran \cite{ErdosTuran}. As an example model, take $f$ as previously defined, but with $L\geq 1$ fixed and with $n\to \infty$. We call this model the \emph{restricted box model}. 

It is a folklore conjecture that in the restricted box model, $f$ is irreducible and $G_f=S_n$ asymptotically almost surely, as $n\to \infty$ (conditioning on $f(0)\neq 0$, of course), cf.\ \cite{OdlyzkoPoonen}. Recently there has been progress on this problem based on the methods developed by Konyagin \cite{Konyagin}:
Bary-Soroker and Kozma \cite{Bary_Soroker_2020} proved that if the length of the interval is divisible by at least 4 distinct primes, then $f$ is irreducible and $A_n\leq G_f$ asymptotically almost surely, provided $f(0)\neq 0$. Breuillard and Varj\'u \cite{Breuillard_2019} proved the same for any $L\geq 1$ assuming the Riemann hypothesis for Dedeking zeta-functions for a certain family of number fields. In fact, they consider a  more general model, where the coefficient are i.i.d.\ with an arbitrary finite support law of distribution. Bary-Soroker, Koukoulopoulos, and Kozma \cite{bary2023irreducibility}  also deal with general measures. Their results are unconditional, and the coefficients are not required to be identical. In particular, in the restricted box model they  prove that $\PP(A_n\leq G_f|f(0)\neq 0)\to 1$, $n\to \infty$ if the interval is of length $\geq 35$ and $\liminf_{n\to \infty} \PP(A_n\leq G_f)>0$ if the length of the interval is between  $2$ to $34$, see \cite[Theorem 6]{bary2023irreducibility}.

As in the large box model, the most challenging case is $G_f=A_n$. Unlike the large box model, in the restricted box model none of the results give that $\mathbb{P}(G_f=A_n)\to 0$, hence it is open whether $\PP(G_f=S_n)\to 1$ as $n\to \infty$. 

The goal of this paper is to get as close as we can to the restricted box model. We get that $G_f=S_n$ asymptotically almost surely, as $L \to \infty$, uniformly in $n$. In particular, $L$ may grow arbitrarily slowly with respect to $n$.

\begin{theorem}\label{thm:Gal_Sn}
    Let $L$ and $n$ be positive integers
    and let
    \begin{equation*}
        f = X^n + \sum_{k=0}^{n-1} \zeta_k X^k
    \end{equation*}
    be a random polynomial, where $\zeta_k$ are chosen independently and identically distributed, taking values uniformly in  $[-L,L]\cap \mathbb{Z}$. Then, 
    \[
        \lim_{L\to \infty}\mathbb{P}(G_f=S_n) = 1,
    \]
    uniformly in $n$.
\end{theorem}

In the regime $L\geq n^7$, the theorem immediately follows from \eqref{GallagherThm}. The interesting part of the theorem is when $L=L(n)$ tends slowly to infinity as $n$ tends to infinity.

The main difficulty lies in the event $G_f=A_n$. Since $G_f\leq A_n$ if and only if $\disc f$ is a perfect square, provided $\disc (f) \neq 0$, we get that $\PP(G_f=A_n)\leq \PP(\disc(f)=\square)$. Hence, Theorem~\ref{thm:Gal_Sn} follows  from the following theorem, see \S\ref{sec:proofthm1}.

    \begin{theorem}\label{th:uniform_measures}
    For every $\frac12>\delta>0$ 
    there exists $N>0$ 
    such that the following holds: 
    Let $\frac{1}{8}>\varepsilon>0$, 
    let $a,L,n$ be integers such that $n>8$ and $L\geq N$, let $\zeta_0,\zeta_1,\ldots $ be independent random variables taking values uniformly in $[a + 1,a + L]\cap \mathbb{Z}$, and let
    \begin{equation*}
        f = X^n + \sum_{k=0}^{n-1} \zeta_k X^k,
    \end{equation*}
    be the corresponding random polynomial. 
    Then, 
    \begin{equation*}
        \mathbb{P}(\disc(f)=\square) \ll 2^{-(\frac{1}{2}-\delta)  \frac{\log L}{\log \log L}} + \frac{\log L}{\log \log L} \ps{\frac{2}{(1-\delta)\log L}}^{(\frac14-\varepsilon) n},
    \end{equation*}
    where the implied constant is absolute.
    \end{theorem}

    The proof of Theorem~\ref{th:uniform_measures} is based on harmonic analysis and on bounds for exponential-M\"obius sums over function fields \cite{BienvenuLe,Porrit_Mu_Exp_Sum}. Our method allows us to prove this theorem for general measures that satisfy several conditions, (see Proposition~\ref{prop:general_result}). We  deduce Theorem~\ref{th:uniform_measures} from Proposition~\ref{prop:general_result} in \S\ref{sec:proofthm2}. 

    \section*{Acknowledgments}
    The authors thank Zeev Rudnick for a beneficial conversation regarding exponential sums.
    
    This research was supported by the Israel Science Foundation (grant no. 702/19).

\section{Square discriminant for general measures}
   \subsection{Harmonic analysis over finite fields}
    We introduce the notation and basic results needed to prove the main theorem. We restrict to prime fields for simplicity of notation. 
    For a prime $p$, let $\FF_p$ be the finite field with $p$ elements  and $\FF_p[T]$ the polynomials ring over $\FF_p$. 
    We denote by $\mathcal{M}_{p,n} \subseteq \mathcal{M}_p\subset \FF_p[T]$ the subsets of monic polynomials of degree $n$ and of all monic polynomials, respectively. 
    Let $\FF_p((T^{-1}))$ be the field of Laurent series of the form $\xi = \sum_{-\infty}^N c_j T^j$, $N \in \ZZ$,  $c_j \in \FF_p$ and let
    \begin{equation*}
        \mathbb{T}_p = \FF_p((T^{-1}))/\FF_p[T] .
    \end{equation*}
    Each element in $\mathbb{T}_p$ has a unique representative of the form $\sum_{j < 0} c_j T^j$, $c_j \in \FF_p$.
    Let $\res_p\colon \mathbb{T}_p \rightarrow \FF_p$ be the additive function defined by $\res_p(\xi) = c_{-1}$.
    In the classical analogy between $\ZZ$ and $\FF_p[T]$, $\mathcal{M}_p$ plays the role of $\ZZ_{>0}$, and $\mathcal{M}_{p,n}$ the role of $[x,2x]\cap \ZZ$, with $\log x$ corresponding to $n$. 
    The analog of $\mathbb{R}$ is $\FF_p((T^{-1}))$ and of $\mathbb{T} = \mathbb{R}/\ZZ$ is $\mathbb{T}_p$. 
    Let 
    \[
        e(z) = e^{2\pi i z} \qquad \mbox{and}\qquad e_p(\xi) = e(\res_p(\xi)/p),
    \]
    for $z\in \CC$ and $\xi\in \mathbb{T}_p$. The latter is well defined. 
    We define the Fourier transform of $\eta\colon \mathcal{M}_{p,n}\to \CC$ to be 
    \[
        \widehat{\eta}(\xi) = \sum_{G\in \calM_{p,n}} \eta(G)e_p(\xi G) , \qquad \xi \in \mathbb{T}_p.
    \]
    This is the analogue of the discrete Fourier transform in the classical setting. 
    
    Our method necessitates considering several primes simultaneously. Let $\calP$ be a finite set of primes and $P=\prod_{p\in \calP}p$.
    Let 
    \[
        \FF_\calP = \prod_{p \in \calP} \FF_p, \quad \FF_{\mathcal{P}}((T^{-1})) = \prod_{p\in\mathcal{P}} \FF_p((T^{-1})), \quad \mbox{and}\quad \mathcal{M}_{\calP,n} = \prod_{p \in \mathcal{P}} \mathcal{M}_{p,n}.
    \]
    Denote the elements of $\mathcal{M}_{\calP,n}$ and of $\mathbb{F}_{\calP}((T^{-1}))$ by $\bF = (F_p)_{p\in\calP}$ and $\boldsymbol{\xi} = (\xi_p)_{p\in \calP}$, respectively. 
    Then, we set
    \begin{align*}
        \psi_\calP(\bxi) = \sum_{p \in \calP} \frac{\res(\xi_p)}{p} \mod 1, &&
         \be(\bxi) 
         = e(\psi_\calP(\bxi)) = \prod_{p\in \calP} e_p(\xi_p).
    \end{align*}
    Then, for $\eta\colon \MM_{\calP,n} \rightarrow \mathbb{C}$ we define $\widehat{\eta}\colon \mathbb{F}_\calP((T^{-1})) \rightarrow \mathbb{C}$ by
    \begin{equation*}
        \widehat{\eta}(\bxi) = \sum_{\bG\in \mathcal{M}_{\calP,n}} \eta(\bG) {\be}(\bxi \bG).
    \end{equation*}

We have the following three classical formulas. The first one is orthogonality of characters: for $\bF\in \MM_{\calP,n}$,
\begin{equation}\label{eq:ortho_charac}
     \sum_{\bG \in \MM_{\calP,n}} e_{\calP} (T^{-n}\bF \bG) = \begin{cases}
            P^n & F_p = T^n, \forall p\in \calP   \\
            0 & \textnormal{otherwise}.
        \end{cases}
\end{equation}
The second is the Fourier inversion formula saying that if $\eta\colon \MM_{\calP,n} \rightarrow \mathbb{C}$, then
    \begin{equation}\label{lm:fourier_inversion}
        \eta(\bF) = \frac{1}{P^n} \sum_{\bG \in \MM_{\calP,n}} \widehat{\eta}(T^{-n}\bG)\be(-T^{-n}\bG \bF),
    \end{equation}
and the last is the Parserval-Plancherel theorem, which for real functions $\eta,\zeta\colon \MM_{\calP,n} \rightarrow \mathbb{R}$ gives that
\begin{equation}\label{lm:Parseval}
        \sum_{\bF \in \MM_{\calP,n}} \eta(\bF) \zeta(\bF) = \frac{1}{P^n} \sum_{\bG \in \MM_{\calP,n}} \widehat{\eta}(T^{-n} \bG) \widehat{\zeta}(-T^{-n} \bG).
    \end{equation}
Since these are so classical we omit their proofs.

\subsection{Main technical result}
Let $f$ be a random monic polynomial of degree $n$ with coefficients in $\ZZ$.
The pushforward defines a probability measure $\PP_{\calP}$ on $\FF_\calP[X]$, supported on  $\mathcal{M}_{\calP,n}$, and an expectation function $\mathbb{E}_{\calP}$: 
\begin{equation}\label{eq:P_P}
    \PP_\calP(\bF) := \PP\left(\bigcap_{p\in\calP} \{f_p = F_p\}\right) \quad \mbox{and}\quad \mathbb{E}_{\calP}(\eta)=\mathbb{E}_{\calP}(\eta(\bF)) = \sum_{\bF\in \mathcal{M}_{\calP,n}} \PP_{\calP}(\bF) \eta(\bF).
\end{equation}
The M\"{o}bius function $\mu_p$ on $\FF_p[T]$ is defined by
\[
\mu_p(F_p)=\begin{cases}
    (-1)^r &   F_p \mbox{ is a product of $r$ distinct irreducible polynomials}\\
    0& F_p \mbox{ is not squarefree.}
    \end{cases}
\]

\begin{proposition}\label{prop:general_result}
     Let $0<\alpha$, $1\leq \gamma<4/3$, and $0<c< \frac{4-3\gamma}{4\gamma}$.
     Then there exists $C>0$ such that the following holds.  Let $(\lambda_k)_{k=0}^{\infty}$ be a sequence of probability measures on $\ZZ$, let  $(\zeta_k)_{k=0}^\infty$ be a sequence of independent random variables with $\zeta_k$ distributing according to $\lambda_k$, $k=0,1,\ldots$, and let 
    \(
        f = X^n + \sum_{k=0}^{n-1} \zeta_k X^k
    \)
    be the corresponding random polynomial.
    Let $\calP$ a finite set of primes and let $\omega\colon \calP \to \mathbb{R}_{\geq 0}$ be a function. 
    Assume that
    \begin{enumerate}[label=(\ref*{prop:general_result}.\arabic*), ref=(\ref*{prop:general_result}.\arabic*)]
        \item $\min\calP\geq C$, \label{cond:no_small_primes}
        \item $\sum_{\substack{\bG \in \MM_{\calP,n}}} |\widehat{\PP}_\calP(T^{-n}\bG)|^{\gamma} \leq \alpha^{\gamma n}$, and \label{cond:fourier_norm_bd}
        \item for every subset $\mathcal{Q} = \{p_1,\ldots,p_r\} \subseteq \calP$, we have 
        \[
        \Bigg|
        \sum_{1 \leq i_1,\ldots,i_r \leq \frac{n}{2} } \underset{\substack{D_k \in         \MM_{k,i_k} \\ \forall 1\leq k \leq r}}
        {\sum \dots \sum}\PP(D_1^2|f_{p_1},\ldots,D_r^2|f_{p_r})\prod_{1\leq m \leq r}\mu_{p_m}(D_m) 
             \Bigg|
            \leq \prod_{1\leq m\leq r}\omega(p_m)^{-1} .
        \]
         \label{cond:sq_divisors}
    \end{enumerate}
    Then, 
    \begin{equation}\label{assert1}
        \mathbb{P}(\disc(f) =\square) \leq 
        \prod_{p\in\calP}\Big(1+\frac{1}{2p^{c n}}\Big) -1+ \frac{1}{2^{\#\calP}} \cdot \prod_{p \in \calP}(1+\omega(p)^{-1}).
    \end{equation}
\end{proposition}

    The proof of Proposition~\ref{prop:general_result} is given in \S\ref{proof:Prop}. 
    It is based on the following two ingredients.  
    
    Let $q$ be an odd prime power, let $\FF_q$ be the finite field with $q$ elements, $\chi_q$ the quadratic character of $\FF_q$, and $\mu$ the M\"obius function. The first is the  formula of Stickelberger and Swan:
    \begin{equation*}
        \chi_q(\disc(F)) = (-1)^{\deg F} \mu (F),
    \end{equation*}
    see, for example, Theorem 6.68 in \cite{Algebraic_Coding_Theory} and the discussion after its proof. 
    Thus, 
    \begin{equation}\label{cor:sq_disc_indicator}
        \mathbbm{1}_{\disc(F) = \square} = \frac{1+(-1)^{\deg F} \mu(F) + \mathbbm{1}_{\mu(F) = 0}}{2},
    \end{equation}
    where $\mathbbm{1}_{\mu(F) = 0} = 1-\mu^2(F)$ is the indicator function of non-squarefree polynomials. The second is a bound 
    on the $L_{\infty}$ norm of $\widehat{\mu}_q$ proved independently by Porrit \cite{Porrit_Mu_Exp_Sum} and Bienvenue and L\^e \cite{BienvenuLe}:
\begin{theorem}\label{thm:mu_exp_sum}
    For every $0<\varepsilon < 1/4$ there exists $q_0>0$ such that for every prime power $q \geq q_0$ we have
    \[
        \max_{\vartheta \in \mathbb{T}} |\widehat{\mu}_q(\vartheta) | \leq q^{(\frac{3}{4}+\varepsilon)n}.
    \]
\end{theorem}

A key step in the proof is to evaluate the probability that the discriminant being a square modulo $p$. In particular, we may prove the following result on finite fields, which we state formally as it may be interesting on its own.
\begin{proposition}
\label{thm:result_over_Fq}
    Let $1\leq \gamma < 4/3$, $\alpha>0$ and $0<c< \frac{4-3\gamma}{4\gamma}$. Then, there exists $C>0$ such that the following holds. Let $q\geq C$ be a prime power, let $(\lambda_k)_{k=0}^{\infty}$ be a sequence of probability measures on $\FF_q$, let  $(\zeta_k)_{k=0}^\infty$ be a sequence of independent random variables with $\zeta_k$ distributing according to $\lambda_k$, $k=0,1,\ldots$, and let 
    \(
        F = X^n + \sum_{k=0}^{n-1} \zeta_k X^k
    \)
    be the corresponding random polynomial.
    Assume that
    \begin{enumerate}
        \item $\sum_{\substack{F \in \MM_{q,n}}} |\widehat{\PP}(T^{-n}F)|^{\gamma} \leq \alpha^{\gamma n}$\label{cond:fourier_norm_bd_fq}\label{cond: finite field 1}, and
        \item $\sum_{i=1}^{\lfloor n/2 \rfloor}\sum_{D \in \mathcal{M}_{q,i}} \mu(D) \cdot \PP(D^2 | F)\leq \omega_q^{-1}$ 
         for some  $\omega_q \in  \mathbb{R}_{\geq 0}$. \label{cond:divisor_cond_fq}
    \end{enumerate}
    Then,
    \begin{equation*}
        \left | \PP(\disc(F)=\square) - \frac{1}{2} \right | \leq \frac{1}{2q^{c n }} + \frac{1}{2\omega_q}.
    \end{equation*}
\end{proposition}

\section{General measures} 
\subsection{Number theory auxiliary results}
For the reader's convenience, we recall some well-known results from number theory that shall be used in subsequent sections. 
Mertens' second theorem says that 
\begin{equation}
    \sum_{p \leq x} \frac{1}{p} = \log \log x + M + O(1/\log x), \label{eq:Merten's}
\end{equation}
where $M=0.261\ldots $ is  the Meissel–Mertens constant. Thus, 
\[\
\sum_{z<p\leq 2z}\frac{1}{p} = \log\log 2z-\log \log z +O(1/\log z) \ll 1,
\]
for all $z>1$. Hence, if $\omega(p)\geq C^{-1} p$ for all $ z< p\leq 2z$, then
\begin{equation}\label{cor:Mertens}
    \prod_{z<p\leq 2z }(1+\omega(p)^{-1})  \leq \exp\left(C\sum_{z<p\leq 2z}  \frac1p  \right) \ll 1.
\end{equation}
The prime number theorem has the following two classical formulations 
\begin{alignat}{2}
    \pi(x) &:= \sum_{p\leq x}1 &&= \frac{x}{\log x} + O(x/ (\log x )^2),\label{eq:pnt} \\
    \vartheta(x) &:=\sum_{p \leq x} \log p &&= x + O(x/\log x).\label{eq:chebyshev's pnt}
\end{alignat}
We may deduce the following bound for $\alpha>1$:
\begin{equation}\label{eq:boundingconvergent}
    \prod_{z<p\leq 2z}(1+Cp^{-\alpha}) \leq \exp\left(C'\frac{z}{z^{\alpha}\log z}\right)=1+O_\alpha\left(\frac{z}{z^{\alpha}\log z }\right). 
\end{equation}

\subsection{The indicator function of non-squarefrees}
The goal of this section is to provide a formula for the expectation value of the indicator function of non-squarefrees modulo several primes. 

Let $f$ be a random monic polynomial over $\ZZ$ with independent coefficients, such as in Proposition~\ref{prop:general_result}.
For a prime $p\in \ZZ$, let $f_p \in \FF_p[X]$ denote the polynomial one gets by reducing the coefficients of $f$ modulo $p$. Similarly, if $\mathcal{P}$ is a set of primes, then we denote $f_\calP := (f_p)_{p\in\calP} \in \MM_{\calP,n}$ and view it as a random variable.

For a subset $\calQ \subseteq \calP$, we extend  multiplicatively the definitions of the M\"obius function and the of indicator function of non-squarefrees: For $\bF\in \calM_{\calP}$, we define 
\begin{equation}
\begin{split}
    {\mu}_\calQ(\bF) &:= \prod_{p \in \calQ} \mu_p(F_p),\label{eq:big_mueta} \\
    {\eta}_\calQ(\bF) & := \prod_{p \in \calQ} \mathbbm{1}_{\mu_p(F_p)=0}.
\end{split}
\end{equation}
By the M\"obius inversion formula applied to the squareful part of $F\in \mathcal{M}_{p,n}$ one gets
\begin{equation}\label{eq:mu^2 identity}
    \mu_p^2(F) = \sum_{D^2|F}\mu_p(D),
\end{equation}
so since $\mathbbm{1}_{\mu_q(F)=0} = 1-\mu^{2}(F)$, we conclude that
\begin{equation}\label{eq:expandeta}
    \eta_{\calQ}(\bF) =  \prod_{p\in\calQ} (1-\sum_{D_p^2\mid F_p}\mu_p(D_p)) =(-1)^{\#\calQ}\prod_{p\in\calQ}\sum_{\substack{D_p^2\mid F_p\\D_p\neq 1}}\mu_p(D_p).
\end{equation}

\begin{lemma}\label{lm:mu=0 formula}
    Let $F$ be a random variable taking values in $\calM_{p,n}$ and let $E$ be an event. Then, 
    \begin{equation*}
        \PP(E\cap \{\mu_p(F)=0\}) = -\sum_{i=1}^{\lfloor n/2 \rfloor} \sum_{D \in \MM_{p,i}} \mu_p(D) \cdot \PP(E,D^2|F).
    \end{equation*}
\end{lemma}

\begin{proof}
    Apply \eqref{eq:expandeta} with $\calQ=\{p\}$ to get 
    \[
    \begin{split}
    \PP(E \cap \{\mu_q(F)=0\}) 
    &= \EE(\ind_E\eta_p(F)) 
    = -\sum_{\substack{D\neq 1}}\mu_q(D)\EE( \mathbbm{1}_{E,D^2| F}).
    \end{split}
    \]
    We are done, since $\mathbb{E}(\mathbbm{1}_{E,D^2| F} )= \mathbb{P}(E,D^2\mid F)$ and $\deg D\leq \frac n2$ as $D^2\mid F$.
\end{proof}

\begin{proposition}
    \label{prop:not_sq_free_eq}
    Let $f$ be chosen as in Proposition~\ref{prop:general_result}, and let $\mathcal{Q} = \{p_1,\ldots ,p_r\}\subseteq \calP$ be  finite sets of prime numbers. Then,
    \begin{align*}
        \mathbb{E}_{\calP}(\eta_{Q})
        =(-1)^r \sum_{1 \leq i_1,\ldots,i_r \leq \lfloor n/2 \rfloor} \underset{\substack{D_k \in \MM_{p_k,i_k} \\ \forall 1\leq k \leq r}}{\sum \dots \sum}\prod_{1\leq m\leq r} 
        \mu(D_m) \cdot \PP_{\calP}(D_1^2|F_{p_1},\ldots,D_r^2|F_{p_r}).
    \end{align*}
\end{proposition}

\begin{proof}
    Since $\mathbb{E}_{\calP}(\eta_{Q}) = \mathbb{P}_{\calP}(\mu_{p_1}(f_{p_1})=0,\ldots, \mu_{p_r}(f_{p_r})=0)$, we may 
    apply Lemma~\ref{lm:mu=0 formula} inductively, to conclude the proof. 
\end{proof}

\subsection{Proof of Proposition~\ref{prop:general_result}}
\label{proof:Prop}
We choose $C$ to be sufficiently large depending only on $\alpha$, $\gamma$, and $c$, with exact value to be determined in the course of the proof.
Let $s=\#\calP$ and $\nu=\min \calP$. 

The condition $\disc(f)=\square$ implies that $\disc(f_p)=\square$ for all $p\in \calP$, since $\disc(f_p) \equiv \disc(f)\mod p$. Recalling  \eqref{eq:P_P} and writing $\bF = (F_p)_{p\in \calP}\in \mathcal{M}_{\calP,n}$, we get  that 
\begin{equation}\label{eq:bound_prob}
\begin{split}
    \mathbb{P}(\disc (f) = \square) 
        &\leq \PP_\calP\Big(\bigcap_{p \in \calP} \disc(F_p) = \square\Big) = \mathbb{E}_{\calP}\Big(\prod_{p\in \calP}\mathbbm{1}_{\disc(F_p)=\square}\Big)\\ 
        &\stackrel{\eqref{cor:sq_disc_indicator}}{=} 2^{-s}\mathbb{E}_\calP\Big( \prod_{p\in \calP} (1+(-1)^{n} \mu_p(F_p) + \mathbbm{1}_{\mu_p(F_p) = 0})\Big)\\
        &\stackrel{\eqref{eq:big_mueta}}{=} 2^{-s}\Big( \sum_{\calQ \subseteq \calP} \mathbb{E}_\calP (\eta_{\calQ}) + \sum_{\emptyset \neq \calQ \subseteq \calP} (-1)^{n|\calQ|}\mathbb{E}_\calP(\mu_{\calQ} \sum_{\mathcal{R} \subseteq \calP \setminus \calQ } \eta_{\mathcal{R}}) \Big).
\end{split}
\end{equation}
We bound the first summand by Proposition~\ref{prop:not_sq_free_eq} and \ref{cond:sq_divisors}: 
\begin{equation}\label{eq:boundoneta}  
    \Big|\sum_{\calQ \subseteq \calP}\mathbb{E}_\calP (\eta_{\calQ})\Big|=\sum_{\calQ \subseteq \calP}\mathbb{E}_\calP (\eta_{\calQ}) \leq \sum_{\calQ \subseteq \calP} \prod_{p\in \calQ} \omega(p)^{-1} = \prod_{p\in \calP}(1+\omega(p)^{-1}).
\end{equation}

Let  $\mathfrak S:= 2^{-s}\left|\sum_{\emptyset \neq \calQ \subseteq \calP} (-1)^{n|\calQ|}\mathbb{E}_\calP(\mu_{\calQ} \sum_{\mathcal{R} \subseteq \calP \setminus \calQ } \eta_{\mathcal{R}})\right|$. We have 
\[
    \left|\sum_{\mathcal{R} \subseteq \calP \setminus \calQ} \eta_\mathcal{R}(\bF)\right|=\sum_{\mathcal{R} \subseteq \calP \setminus \calQ} \eta_\mathcal{R}(\bF) \leq 2^{s-|\calQ|}.
\]
Hence,
\begin{equation}\label{eq:bd_S}
    \mathfrak S \leq \sum_{\emptyset \neq \calQ \subseteq \calP} {2^{-|\calQ|}}  \left|\mathbb{E}_{\calP}( \mu_\calQ)
    \right|.
\end{equation}
To this end, fix $\emptyset\neq \calQ \subseteq \calP$. By \eqref{lm:Parseval}, 
\begin{equation}\label{eq:afterParseval}
    \mathbb{E}_{\calP}( \mu_\calQ) = \frac{1}{P^n} \sum_{\bG \in \MM_{\calP,n}} \widehat{\PP}_{\calP}(T^{-n}\bG) \widehat{\mu}_\calQ(-T^{-n}\bG).
\end{equation}
Expanding $\widehat{\mu}_{\calQ}$ by definition gives
\begin{multline*}
    \widehat{\mu}_\calQ(-T^{-n}\bG) =  \sumP{\bH} \mu_\calQ(\bH) \be(-T^{-n}\bG \bH) \\
    = \prod_{p \in \calQ} \ps{\sum_{H_p \in \MM_{p,n}} \mu_p (H_p) e(\psi_p(-T^{-n} G_p H_p))} \cdot \prod_{p \in \calP \setminus \calQ} \ps{\sum_{H_p \in \MM_{p,n}} e(\psi_p(-T^{-n} G_p H_p))}.
\end{multline*}
By \eqref{eq:ortho_charac}, the product on the right vanishes unless $G_p=T^n$ for all $p\in \calP\setminus \calQ$, in which case it equals $P^n/Q^n$, where $Q=\prod_{p\in \calQ}p$. Plugging this in \eqref{eq:afterParseval} gives
\begin{equation}\label{eq:upsilon_calculation}
    \EE_{\calP}(\mu_\calQ) = \frac{1}{Q^n} \sum_{\substack{\bG \in \MM_{\calP,n} \\ \forall p \notin\calQ:\ G_p = T^n}} \widehat{\PP}_{\calP}(T^{-n} \bG) \prod_{p \in \calQ} \widehat{\mu}_p(-T^n G_p).
\end{equation}

Let $\delta = \frac{\gamma}{\gamma-1} > 4$, with $\delta=\infty$ if $\gamma=1$, so that $\frac{1}{\gamma}+\frac1\delta=1$. Applying H\"{o}lder's inequality to \eqref{eq:upsilon_calculation} gives that 
\begin{equation}\label{eq:bound_Upsilon}
    |\EE_\calP(\mu_{\calQ})| \leq \frac{1}{Q^n} \Big(
    \sum_{\substack{\bG \in \MM_{\calP,n} \\ \forall p \notin\calQ:\ G_p = T^n}} |\widehat{\PP}_\calP(T^{-n} \bG)|^\gamma
    \Big)^{\frac{1}{\gamma}} 
    \Big(
    \sum_{\substack{\bG \in \MM_{\calP,n} \\ \forall p \notin\calQ:\ G_p = T^n}} \prod_{p \in \calQ} |\widehat{\mu}_p(-T^n G_p)|^{\delta}
    \Big)^{\frac{1}{\delta}}.
\end{equation}
The first term is at most $\alpha^n$ by \ref{cond:fourier_norm_bd}. 
By the choice of $c$, we have that $\frac{1}{4}-\frac{1}{\delta}-c>0$. Hence we may take $\varepsilon_0=\varepsilon_0(\gamma,\varepsilon)>0$ so small such that $\frac{1}{4}-\frac{1}{\delta}-\varepsilon_0>c$. Applying Theorem~\ref{thm:mu_exp_sum} with $\varepsilon_0$ yields $q_0=q_0(\gamma,c)$ such that if $p>q_0$, then $\|\widehat{\mu}_p\|_{\infty}\leq p^{(3/4+\varepsilon_0)n}$. So,
\[
    |\EE_{\calP}(\mu_\calQ)| \leq \frac{\alpha^n}{Q^n} 
    \Big(Q^n \cdot \prod_{p\in\calQ} p^{(3/4+\varepsilon_0)n\delta}
    \Big)^{\frac{1}{\delta}} \leq Q^{-un},
\]
where $u=\frac14-\frac{1}{\delta} -\varepsilon_0-\frac{\log\alpha}{\log Q}$. As $Q\geq p\geq C$, if we take $C\geq q_0$ and to be sufficiently large so that $u>c$,
then 
\[
    |\EE_{\calP}(\mu_\calQ)| \leq Q^{-c n}.
\]
We plug this into \eqref{eq:bd_S} to get 
\begin{align*}
    \mathfrak{S} \leq \sum_{\emptyset \neq \calQ \subseteq \calP}2^{-|\mathcal{Q}|} Q^{-c n}  = \prod_{p\in\calP}\left(1+\frac{1}{2p^{c n}}\right) - 1 .
\end{align*} 
Plugging this and \eqref{eq:boundoneta} into \eqref{eq:bound_prob} finishes the proof.
\qed

\subsection{Proof of Proposition~\ref{thm:result_over_Fq}}
    Similarly to the proof of Proposition~\ref{prop:general_result}, we calculate the probability by
    \begin{equation}\label{eq:bound_prob_finite_field}
\begin{split}
    \mathbb{P}(\disc (F) = \square) 
        &= \mathbb{E}_q \Big( \mathbbm{1}_{\disc(F)=\square}\Big)\\ 
        &\stackrel{\eqref{cor:sq_disc_indicator}}{=} \frac{1}{2} \mathbb{E} \Big(1+(-1)^{n} \mu_q(F) + \mathbbm{1}_{\mu(F) = 0}\Big)\\
        &= \frac{1}{2} + \frac{(-1)^n}{2} \mathbb{E}(\mu_q(F)) + \frac{1}{2} \mathbb{E} (\mathbbm{1}_{\mu_q(F) = 0}) 
\end{split}
\end{equation}
    We bound each of the terms separately. Similarly to \eqref{eq:bound_Upsilon}, by H\"{o}lder's inequality we have
    \begin{equation*}
        |\EE(\mu_q)| \leq \frac{1}{q^n} \Big(
    \sum_{G \in \MM_{q,n}} |\widehat{\PP}(T^{-n} G)|^\gamma
    \Big)^{\frac{1}{\gamma}} 
    \Big(
    \sum_{G \in \MM_{q,n}} |\widehat{\mu_q}(-T^n G)|^{\delta}
    \Big)^{\frac{1}{\delta}}.
    \end{equation*}
    The first term is at most $\alpha^n$ by Condition \ref{cond: finite field 1}. 
    In a similar manner as in \eqref{eq:bound_Upsilon}, we get that if we pick $C$ sufficiently large  relatively to $\alpha, \gamma$ and $c$, then 
    \[
        |\EE(\mu_q)| \leq q^{-c n}.
    \]
    
    For the second term, using Lemma~\ref{lm:mu=0 formula} and Condition~\ref{cond:divisor_cond_fq}, we have
    \begin{align*}
         \left | \mathbb{E}_q (\mathbbm{1}_{\mu_q(F) = 0}) \right |
         =  \left | \sum_{i=1}^{\lfloor n/2 \rfloor} \sum_{D \in \MM_i} \mu_q(D) \cdot \PP(D^2|F) \right | \leq \omega_q^{-1},
    \end{align*}
    which completes the proof.
\qed

\section{Proof of the main theorems}

\subsection{Theorem~\ref{th:uniform_measures} implies Theorem~\ref{thm:Gal_Sn}}
\label{sec:proofthm1}
We prove that Theorem~\ref{th:uniform_measures} implies a slightly more general version of Theorem~\ref{thm:Gal_Sn}:
\begin{theorem}\label{thm:Gal_Sn'}
        Let $L,n$ be positive integers, let $a \in \ZZ$ and let
    \begin{equation*}
        f = X^n + \sum_{k=0}^{n-1} \zeta_k X^k
    \end{equation*}
    be a random polynomial, where $\zeta_k$ are chosen independently and identically distributed, taking values uniformly in  $[a+1,a+L]\cap \mathbb{Z}$. Then, 
    \[
        \lim_{L\to \infty}\mathbb{P}(G_f=S_n) = 1,
    \]
    uniformly on all pairs $(n,a)$ with $n
    \geq 1$, $a\in \mathbb{Z}$, and such that if $n^7>L$, then  $|a|\leq \frac12e^{n^{1/3}}$.
\end{theorem}
Theorem~\ref{thm:Gal_Sn} follows from Theorem~\ref{thm:Gal_Sn'} immediately by replacing $L$ by $2L+1$ and setting $a=-L$.

    \begin{proof}[Proof of Theorem~\ref{thm:Gal_Sn'}]
        Let $f$ be as in Theorem~\ref{thm:Gal_Sn} and let 
    $p=\mathbb{P}(G_f\neq S_n)$. We need to prove that $p\to 0$ as $L\to \infty$ uniformly on $(n,a)\in \{ (n,a): n^7>L\Rightarrow |a|\leq \frac{1}{2}e^{n^{1/3}}\}$. 
    By \eqref{GallagherThm}, $p\to 0$ uniformly as $L\geq n^7$.

    To this end, assume $n^7>L$. 
    Since $|a|\leq \frac{1}{2}e^{n^{1/3}}$, we get that $[a+1,a+L]\subseteq [-e^{n^{1/3}},e^{n^{1/3}}]$, for $L$ sufficiently large. Hence Condition (a) of \cite[Theorem 8]{bary2023irreducibility} is satisfied.
    Condition (b) is satisfied with $P=210$ since we may assume that $L\geq 33,730$ (the details appear in the proof of \cite[Theorem 1(a)]{bary2023irreducibility}). 
    Hence, we may apply \cite[Theorem 8]{bary2023irreducibility} to get that $\mathbb{P}(A_n\not\leq G_f) =O(n^{-c})$, with $c>0$ absolute.

    Finally, by Theorem~\ref{th:uniform_measures}, we get that, in this regime,
    \begin{multline*}
        p\leq \mathbb{P}(A_n\not\leq G_f) + \mathbb{P}(\disc f=\square)
        \\ \ll n^{-c} + 2^{-(\frac{1}{2}-\delta)  \frac{\log L}{\log \log L}} + \frac{\log L}{\log \log L} \ps{\frac{2}{(1-\delta)\log L}}^{(\frac14-\varepsilon) n} .
    \end{multline*}
    Thus, $p\to 0$ uniformly as  $n^7>L\to \infty$, and this concludes the proof. 
    \end{proof}

\subsection{Preliminaries for the proof of Theorem~\ref{th:uniform_measures}}
We will apply Proposition~\ref{prop:general_result} to prove Theorem~\ref{th:uniform_measures}. The following two lemmas are needed  to establish \ref{cond:fourier_norm_bd} and \ref{cond:sq_divisors}. 

We start with a simple bound:
Let $\xi$ be a random variable distributed uniformly on an interval $[a+1,a+L]\cap \mathbb{Z}$ of length $L$ and let $u,d\in \mathbb{Z}$ with $d>0$. Write $L=qd+r$ with $0\leq r<d$, then  
\[
    \mathbb{P}(\xi \equiv u \mod d) = \frac{\#\{ v\in [a+1,a+L]\cap \mathbb{Z} : v\equiv u\mod d\} }{L} =\frac{q+\alpha}{L},
\]
where $\alpha =\#\{ v\in [a+qd+1,a+L]:v\equiv u\mod d\}\in \{0,1\}$. Thus
\begin{equation}\label{eq:totaldeviation}
    \frac{1}{d}-\frac{1}{L}\leq \mathbb{P}(\zeta \equiv u \mod d) \leq \frac{1}{d}+\frac{1}{L}.
\end{equation}

\begin{lemma}\label{lm:L1_estimation}
    Let $f$ be as in Theorem~\ref{th:uniform_measures}, let $\calP$ be a finite set of primes, and let $P:=\prod_\calP p$. Then,
    \[
        \sum_{\bF \in \MM_{\calP,n}} |\widehat{\PP}_\calP(T^{-n}\bF)| \leq \ps{1+\frac{P(P-1)}{L}}^n.
    \]
\end{lemma}

\begin{proof}
    For a polynomial $H \in \mathbb{F}_p[T]$, we denote by $H^i$ its $i$-th coefficient, i.e. $H = \sum_{i=0}^{\deg H}H^{i}T^{i}$. Since $(\zeta_i)_{i=0}^{n-1}$ are independent, 
    \begin{align*}
        \left|\widehat{\PP}_\calP(T^{-n}\bF)\right| 
        &= 
        \left| \sumP{G} \PP_\calP(\bG) e_{\calP}(T^{-n}\bF \bG)\right| 
        \\ &= 
        \left| \sumP{G} \prod_{i=0}^{n-1} \PP(\zeta_i \equiv G_p^i \mod{p},\ \forall {p\in\calP} ) \prod_{p\in \calP}\prod_{i=0}^{n-1} e(\psi_p(G_p^i F_p^{n-1-i}))\right| 
        \\ &= 
        \prod_{i=0}^{n-1} \left|\sum_{(G_p^{i})_{p\in \calP} \in \FF_\calP} \PP( \zeta_i \equiv G^i_p \mod{p},\ \forall {p\in\calP}) \prod_{p\in\calP}e(\psi_p(G_p^i F_p^{n-1-i}))\right|.
    \end{align*}
    To this end, fix $0\leq i\leq n-1$. If $F_p^{n-1-i}=0$ for all $p$, then 
    \[
        \sum_{(G_p^{i})_{p\in \calP} \in \FF_\calP} \PP( \zeta_i \equiv G^i_p \mod{p},\ \forall {p\in\calP})\prod_{p\in\calP} e(\psi_p(G_p^i F_p^{n-1-i}))=1,
    \]
    as a sum over all probabilities. Otherwise, there exists a $p$ such that $F_p^{n-1-i}\neq 0$. 
    Hence,
    \begin{multline*}
        \left|\sum_{(G_p^{i})_{p\in \calP} \in \FF_\calP} \PP( \zeta_i \equiv G^i_p \mod{p},\ \forall {p\in\calP})\prod_{p\in\calP} e(\psi_p(G_p^i F_p^{n-1-i})) \right| \\
        \stackrel{\eqref{eq:ortho_charac}}{=}\Bigg|\sum_{(G_p^{i})_{p\in \calP} \in \FF_\calP} \PP( \zeta_i \equiv G^i_p \mod{p},\ \forall {p\in\calP})\prod_{p\in\calP} e(\psi_p(G_p^i F_p^{n-1-i})) \\
        - \sum_{(G_p^{i})_{p\in \calP} \in \FF_\calP} \frac{1}{P}\prod_{p\in\calP} e(\psi_p(G_p^i F_p^{n-1-i}))\Bigg|
        \stackrel{\eqref{eq:totaldeviation}}{\leq} 
        \frac{P}{L}
    \end{multline*}
    Writing $k(\bF)=\#\{0\leq i\leq n-1: \exists p\in\calP,\ F_p^{n-1-i}\neq 0\}$, we get that
    \[
        \prod_{i=0}^{n-1} \left|\sum_{(G_p^{i})_{p\in \calP} \in \FF_\calP} \PP( \zeta_i \equiv G^i_p \mod{p},\ \forall {p\in\calP}) \prod_{p\in\calP}e(\psi_p(G_p^i F_p^{n-1-i}))\right| \leq \left(\frac{P}{L}\right)^{k(\bF)}.
    \]
    Hence,
    \begin{align*}
        \sum_{\bF \in \MM_{\calP,n}} |\widehat{\PP}_{\calP}(T^{-n}\bF)| &\leq \sum_{k=0}^n \binom{n}{k} \ps{\frac{P}{L}}^k (P -1)^{k}
        = \ps{1+\frac{P(P-1)}{L}}^n,
    \end{align*}
    as needed.
\end{proof}

\begin{lemma}\label{lm:convenient cond3} Assume the setting of Proposition~\ref{prop:general_result}. 
    Let $h\colon\calP \to \mathbb{R}_{\geq 0}$
    a function such that $h(p)^2>p$,  for all $p \in \calP$. 
    Assume that for all $d | P$, $\alpha\in\ZZ$, and $k \geq 0$ we have $\PP_\calP(\zeta_k = \alpha \mod d) \leq \prod_{p\mid d}h(p)^{-1}$. Then, $\omega(p) = \frac{h(p)^2}{p}-1$ 
    satisfies \ref{cond:sq_divisors}.
\end{lemma}

\begin{proof}
    By assumption and the Chinese Remainder Theorem
    \begin{multline*}
        \sum_{1 \leq i_1,\ldots,i_r \leq \lfloor n/2 \rfloor} \underset{\substack{D_k \in         \MM_{p_k,i_k} \\ \forall 1\leq k \leq r}}{\sum \dots \sum}\ 
            \PP(D_1^2|f_{p_1},\ldots,D_r^2|f_{p_r}) 
            \leq \sum_{1 \leq i_1,\ldots,i_r \leq \lfloor n/2 \rfloor} \prod_{k=1}^r \frac{p_k^{i_k}}{h(p_k)^{2i_k}}
            \\
            = \prod_{k=1}^{r} \sum_{i=1}^{\lfloor n/2 \rfloor} \left(\frac{p_k}{h(p_k)^{2}}\right)^i
            \leq \prod_{k=1}^{r} \sum_{i=1}^{\infty} \left(\frac{p_k}{h(p_k)^2}\right)^i
            \leq \prod_{k=1}^r \frac{\frac{p_k}{h(p_k)^2}}{1-\frac{p_k}{h(p_k)^2}}=\prod_{k=1}^r\omega(p_k)^{-1},
    \end{multline*}
    as needed.
\end{proof}

\subsection{Proof of Theorem~\ref{th:uniform_measures}}\label{sec:proofthm2}
    First we show that if $L$ is sufficiently large with respect to $\delta$ then the conditions of Proposition~\ref{prop:general_result} are satisfied with $\lambda_k$ uniformly distributed  on $[a+1,a+L]\cap \mathbb{Z}$, $\alpha=2$, $\gamma=1$, $c=\frac{1}{4}-\varepsilon$, and $\omega(p) = \frac{p-4}{4}$: 
    Let $\calP$ be the set of all primes  $\frac{1-\delta}{2} \log L < p \leq (1-\delta)\log L$ and $P=\prod_{p\in \calP} p$.
    Then, $\min\calP\geq \frac{1-\delta}{2} \log L$, and so \ref{cond:no_small_primes} holds true for $L$ sufficiently large.
    By \eqref{eq:chebyshev's pnt}, $L\geq P(P-1)$ if $L$ is sufficiently large, so 
     by Lemma~\ref{lm:L1_estimation}, we have
    \[
        \sum_{\bF \in \MM_{\calP,n}} |\widehat{\PP}_{\calP}(T^{-n}\bF)| \leq \ps{1+\frac{P(P-1)}{L}}^n \leq 2^n,
    \]
    Hence, \ref{cond:fourier_norm_bd} is satisfied.
    Let $h\colon \calP\to \mathbb{R}_{\geq 0}$ be the function defined by $h(p) = \frac{p}{2}$ so that $\omega(p)=\frac{h(p)^2}{p}-1$.
    For  $d \mid P$ and  $u \in \ZZ$, we have $d\leq P\leq L$, hence 
    \begin{equation*}\label{eq:Prob zeta_i=a mod p}
        \PP(\zeta_i = u \mod d) \stackrel{\eqref{eq:totaldeviation}}{\leq} \frac{1}{d} + \frac{1}{L} \leq \prod_{p \divides d}h(p)^{-1}.
    \end{equation*}
    This implies \ref{cond:sq_divisors} by Lemma~\ref{lm:convenient cond3}.

    Now we may apply Proposition~\ref{prop:general_result} to get that
    \begin{equation}\label{eq:bounddisc}
        \mathbb{P}(\disc(f) =\square) \leq \prod_{p\in \calP}\Big(1+2^{-1}p^{-(\frac14-\varepsilon) n}\Big)-1 + \frac{1}{2^{\# \calP}} \prod_{p\in \calP}(1+\omega(p)^{-1}).
    \end{equation}    
    By \eqref{cor:Mertens} and \eqref{eq:pnt}, we have
    \[
        \frac{1}{2^{\# \calP}} \prod_{p\in \calP}(1+\omega(p)^{-1})\ll    2^{- \frac{1-\delta}{2} \frac{\log L}{\log\log L}},  
    \]
    where the implied constant is absolute. 
    
    By \eqref{eq:boundingconvergent}, (note that $(\frac14 -\varepsilon)n>\frac{n}{8}>1$), we have
    \[
         \prod_{p\in \calP}(1+2^{-1}p^{-(1/4-\varepsilon) n})-1 \ll\frac{\log L}{(\frac{1-\delta}{2}\log L)^{(\frac14-\varepsilon) n}\log(\frac{1-\delta/2}{2}\log L)},
    \]
    where the implied constant is absolute. Since $\log(\frac{1-\delta/2}{2}\log L) = \log\log L+O(1)$, plugging the above bounds into \eqref{eq:bounddisc} completes the proof.\qed


\begin{thebibliography}{10}

\bibitem{Andersonetal}
T.~C. Anderson, A.~Gafni, R.~J. Lemke~Oliver, D.~Lowry-Duda, G.~Shakan, and
  R.~Zhang.
\newblock Quantitative {H}ilbert irreducibility and almost prime values of
  polynomial discriminants.
\newblock {\em Int. Math. Res. Not. IMRN}, (3):2188--2214, 2023.

\bibitem{bary2022probabilistic}
L.~Bary-Soroker, O.~Ben-Porath, and V.~Matei.
\newblock Probabilistic {G}alois theory--the square discriminant case.
\newblock {\em Bull. Lond. Math. Soc.}, in press, expected 2024.

\bibitem{bary2023irreducibility}
L.~Bary-Soroker, D.~Koukoulopoulos, and G.~Kozma.
\newblock Irreducibility of random polynomials: general measures.
\newblock {\em Invent. Math.}, 233(3):1041--1120, 2023.

\bibitem{Bary_Soroker_2020}
L.~Bary-Soroker and G.~Kozma.
\newblock Irreducible polynomials of bounded height.
\newblock {\em Duke Math. J.}, 169(4):579--598, 2020.

\bibitem{Algebraic_Coding_Theory}
E.~R. Berlekamp.
\newblock {\em Algebraic coding theory (revised edition)}.
\newblock World Scientific, 2015.

\bibitem{Bhargava}
M.~Bhargava.
\newblock A proof of van der {W}aerden's conjecture on random {G}alois groups
  of polynomials.
\newblock {\em Pure Appl. Math. Q.}, 19(1):45--60, 2023.

\bibitem{BienvenuLe}
Pierre-Yves Bienvenu and Th\'{a}i~Ho\`ang L\^{e}.
\newblock Linear and quadratic uniformity of the {M}\"{o}bius function over
  {$\Bbb F_q[t]$}.
\newblock {\em Mathematika}, 65(3):505--529, 2019.

\bibitem{BlochPolya}
A.~Bloch and G.~P\'{o}lya.
\newblock On the roots of certain algebraic equations.
\newblock {\em Proc. London Math. Soc. (2)}, 33(2):102--114, 1931.

\bibitem{Breuillard_2019}
E.~Breuillard and P.~P. Varj\'u.
\newblock Irreducibility of random polynomials of large degree.
\newblock {\em Acta Mathematica}, 223(2):195–249, 2019.

\bibitem{Chela}
R.~Chela.
\newblock Reducible polynomials.
\newblock {\em J. London Math. Soc.}, 38:183--188, 1963.

\bibitem{ChowDietmann}
S.~Chow and R.~Dietmann.
\newblock Enumerative {G}alois theory for cubics and quartics.
\newblock {\em Adv. Math.}, 372:107282, 37, 2020.

\bibitem{ChowDietmann2}
S.~Chow and R.~Dietmann.
\newblock Towards van der {W}aerden's conjecture.
\newblock {\em Trans. Amer. Math. Soc.}, 376(4):2739--2785, 2023.

\bibitem{Dietmann_2011}
R.~Dietmann.
\newblock On the distribution of {G}alois groups.
\newblock {\em Mathematika}, 58(1):35--44, 2011.

\bibitem{ErdosTuran}
P.~Erd\H{o}s and P.~Tur\'{a}n.
\newblock On the distribution of roots of polynomials.
\newblock {\em Ann. of Math. (2)}, 51:105--119, 1950.

\bibitem{Gallagher_1976}
P.~X. Gallagher.
\newblock The large sieve and probabilistic {G}alois theory.
\newblock In {\em Analytic number theory ({P}roc. {S}ympos. {P}ure {M}ath.,
  {V}ol. {XXIV}, {S}t. {L}ouis {U}niv., {S}t. {L}ouis, {M}o., 1972)}, volume
  Vol. XXIV of {\em Proc. Sympos. Pure Math.}, pages 91--101. Amer. Math. Soc.,
  Providence, RI, 1973.

\bibitem{Hilbert}
D.~Hilbert.
\newblock Ueber die {I}rreducibilit\"{a}t ganzer rationaler {F}unctionen mit
  ganzzahligen {C}oefficienten.
\newblock {\em J. Reine Angew. Math.}, 110:104--129, 1892.

\bibitem{Kac}
M.~Kac.
\newblock On the average number of real roots of a random algebraic equation.
\newblock {\em Bull. Amer. Math. Soc.}, 49:314--320, 1943.

\bibitem{Knobloch}
H.-W. Knobloch.
\newblock Die {S}eltenheit der reduziblen {P}olynome.
\newblock {\em Jber. Deutsch. Math.-Verein.}, 59:12--19, 1956.

\bibitem{Konyagin}
S.~V. Konyagin.
\newblock On the number of irreducible polynomials with {$0,1$} coefficients.
\newblock {\em Acta Arith.}, 88(4):333--350, 1999.

\bibitem{Kowalski_2008}
E.~Kowalski.
\newblock {\em The large sieve and its applications}, volume 175 of {\em
  Cambridge Tracts in Mathematics}.
\newblock Cambridge University Press, Cambridge, 2008.
\newblock Arithmetic geometry, random walks and discrete groups.

\bibitem{LittlewoodOfford}
J.~E. Littlewood and A.~C. Offord.
\newblock On the number of real roots of a random algebraic equation.
\newblock {\em J. London Math. Soc.}, 13(4):288--295, 1938.

\bibitem{OdlyzkoPoonen}
A.~M. Odlyzko and B.~Poonen.
\newblock Zeros of polynomials with {$0,1$} coefficients.
\newblock {\em Enseign. Math. (2)}, 39(3-4):317--348, 1993.

\bibitem{Porrit_Mu_Exp_Sum}
S.~Porritt.
\newblock A note on exponential-möbius sums over $\mathbb{F}_q[t]$.
\newblock {\em Finite Fields and Their Applications}, 51:298--305, 2018.

\bibitem{VanDerWaerden}
B.~L. van~der Waerden.
\newblock Die {S}eltenheit der reduziblen {G}leichungen und der {G}leichungen
  mit {A}ffekt.
\newblock {\em Monatsh. Math. Phys.}, 43(1):133--147, 1936.

\end{thebibliography}

\end{document}